\documentclass[leqno,11pt]{amsart}

\usepackage{amsfonts,amssymb,amsgen,stmaryrd, amsopn}
\usepackage[leqno]{amsmath} 
\usepackage[T1]{fontenc}
\usepackage[latin1]{inputenc} 
\usepackage[french,english]{babel}
\usepackage{graphicx}
\usepackage{graphpap}
\usepackage{hyperref}
\usepackage{mathabx}

\newtheorem{theoreme}{Theorem}[section]

\newtheorem{proposition}{Proposition}[section]
 \newtheorem{lemme}{Lemma}[section]
 \newtheorem{definition}{Definition}[section]
 \newtheorem{rem}{Remark}[section]
 
 \renewcommand\Im{\mathrm{Im}\,}
\newcommand\R{{\mathbb R}} \newcommand\N{{\mathbb N}}
\newcommand\Z{{\mathbb Z}} 
 
 \renewcommand\S{{\mathbb S}}

 \newcommand{\cqfd}{\mbox{ }
\hfill$\Box$} 

\renewcommand{\Im}{  \text{Im}   }

 \def\cdotv{\raise 2pt\hbox{,}}

 \renewcommand\Im{\mathrm{Im}\,}

\makeatletter
\def\@tvsp{\mathchoice{{}\mkern-4.5mu}{{}\mkern-4.5mu}{{}\mkern-2.5mu}{}}
\def\ltrivert{\left|\@tvsp\left|\@tvsp\left|}
\def\rtrivert{\right|\@tvsp\right|\@tvsp\right|}
\makeatother
\newcommand{\tn}[2]{\ltrivert #1\rtrivert_#2}

 \def\cdotv{\raise 2pt\hbox{,}}

\begin{document}
\title{On the cubic NLS on 3D compact domains}
 \author{Fabrice Planchon}
 \address{Laboratoire J. A. Dieudonn\'e, UMR 7351\\
 Universit\'e de Nice Sophia-Antipolis\\
 Parc Valrose\\
 F-06108 Nice Cedex 02\\
 et Institut universitaire de France}
 \email{fabrice.planchon@unice.fr}
   \thanks{The author was partially supported by A.N.R. grant SWAP}
\date{\today}
 \maketitle
 \begin{abstract}
We prove bilinear estimates for the Schr\"odinger equation on 3D
domains, with Dirichlet boundary conditions. On non-trapping domains,
they match the $\mathbb{R}^3$ case, while on bounded domains they
match the generic boundaryless manifold case. As an application, we
obtain global well-posedness for the defocusing cubic NLS for data in
$H^s_0(\Omega)$, $1<s\leq 3$, with $\Omega$ any bounded
domain with smooth boundary.
\\
\selectlanguage{french}
\begin{center}
  Résumé
\end{center}
On démontre des estimations bilinéaires pour l'équation de
Schr\"odinger sur des domaines tridimensionnels, avec condition de
Dirichlet au bord. Dans le cas non-captant, on retrouve les
estimations connues dans $\mathbb{R}^3$, et sur un domaine borné on
obtient des estimations similaires à celles du cas d'une variété
compacte générique sans bord. Une application est donnée à l'existence
de solutions globales dans $H^s_0(\Omega)$, $1< s \leq 3$,
avec $\Omega$ un domaine borné régulier.
\selectlanguage{english}
\end{abstract}
 \par \noindent
\section{Introduction}
Let $\Omega\subset \R^3$ be a domain with a smooth boundary
$\partial\Omega$, and consider the Schr\"odinger equation
\begin{equation}
\label{nls}
  i\partial_t \phi+\Delta \phi=\varepsilon |\phi|^2 \phi, \text{ with }
  \phi_{|\partial \Omega} =0 \text{ and } \phi_{t=0}=\phi_0\,.
\end{equation}
Our first interest is the linear equation, that is
$\varepsilon=0$. When $\Omega=\R^3$, dispersive properties of
\eqref{nls} are well-understood and play a crucial role in
understanding the nonlinear case $\varepsilon=\pm 1$. One has a large
set of sharp Strichartz estimates (\cite{stri,gv,kt} and among other
things the nonlinear problem is locally well-posed in the Sobolev
space $\dot H^\frac 1 2(\R^3)$ (\cite{cawe90}), and globally
well-posed in the energy space $H^1(\R^3)$ (with a smallness condition
on the $L^2(\R^3)$ norm in the focusing case $\varepsilon=-1$), see
\cite{gv}. On domains, our understanding of dispersion is far from
complete: depending on the geometry of light rays, we have essentially
two cases,
\begin{itemize}
\item $\Omega$ is the exterior of a non trapping obstacle: a
  restricted set of non sharp (but scale-invariant) Strichartz
  estimates is known to hold (\cite{BSS-ext}, see also \cite{plve08} and
  \cite{bgtexterieur}), together with square-function type estimates,
  \cite{ip09}. In the special case of a strictly convex obstacle,
  Strichartz estimates hold as in $\R^3$ (except possibly for the
  endpoint), \cite{oanass};
\item $\Omega$ is a compact domain with smooth boundaries: there, the
  same set of estimates as in the exterior case in known to hold, but
  only at semi-classical time scales. In fact, the estimates in the non
  trapping case are obtained by combining these semi-classical
  estimates with the local smoothing effect, following a strategy
  pioneered by \cite{st}. Thus, the best available Strichartz
  estimates are due to \cite{BSS-ext} (see also \cite{bss08} and
  \cite{an05} for earlier progress) and they exhibit losses with
  respect to scaling (translating to a regularity loss when compared
  to the exterior case). Moreover, when compared to Strichartz
  estimates on compact manifolds without boundaries (\cite{bgt04}), we
  have a restricted set of exponents. As such, one may solve the
  nonlinear cubic equation \eqref{nls} only for $H^{s}(\Omega)\cap
  H^1_0(\Omega)$, $1<s\leq 2$, and this is a local in time construction (combined
  logarithmic losses in the available Strichartz estimates prevent a
  Brezis-Gallou\"et or Yudovitch-like argument). Moreover, counterexamples to the sharp Strichartz
  estimates were constructed in \cite{oanace} (though this leaves a
  large gap to be filled between our current knowledge and these
  counterexamples).
\end{itemize}
On compact, boundary less, manifolds, an alternate strategy was pushed
forward in \cite{bgt-inv,bgt-anens}: Strichartz estimates are replaced
by bilinear $L^2_{t,x}$ estimates, following \cite{Bou-gafa1} which
dealt with rational tori. On specific manifolds (e.g. Zoll manifolds), for which
eigenvalues and eigenfunctions are well-distributed, these bilinear
estimates prove to be more efficient than Strichartz estimates. In
\cite{anton-radial}, such bilinear estimates are established for
radial data on $\Omega=B(0,1)$, and they match the Zoll manifold case
(however, \cite{anton-radial} provides a counter-example to such
optimal estimates in the non radial case), providing near optimal
well-posedness for the cubic NLS. Moreover, we learned after
completion of the present work that the radial quintic NLS may be
addressed (\cite{PauTz}) by tranfering estimates from
the boundary less manifold $\mathbb{S}^3$. In a different direction, bilinear
estimates have proved to be quite useful in the whole space when
dealing with long time behavior of nonlinear solutions (from
\cite{bourgain98} to \cite{cksttannals}). Heuristically, one would expect
such bilinear estimates to hold on the semiclassical time scale for generic
compact manifolds, and this is indeed true, \cite{hani}. 

Our aim is to obtain such estimates on generic domains, with Dirichlet
boundary conditions. In \cite{plve08}, we derived a linear $L^4_{t,x}$
estimate from our general bilinear virial machinery (later matched in
\cite{BSS-ext}, among a larger set of estimates); as a side note,
\cite{plve08} provides an argument allowing to recover \cite{bourgain98}
from the Radon transform estimate, but such an argument, while mostly
based on integration by parts, is still using preservation of the
Fourier support by the linear flow (and the argument is, in some
sense, non local due to the use of the Radon transform). In
\cite{bourbaki}, we provide a different argument, much more local, and
based on a Sobolev trace lemma (such an argument would allow to
recover \cite{hani} through an appropriate microlocalization
procedure); similar arguments appeared independently in \cite{smith1}
(see also \cite{smith2} for further developments) in dealing with
the covariant Schr\"odinger equation. However these procedures still
rely on localizing the Fourier support in a specific direction, which
seems out of reach for the boundary value case. Very recently,
however, \cite{Dodson} obtained a variant of such bilinear estimates
on the exterior of the strictly convex obstacle (such a gradient form
of bilinear estimates already proved useful in
\cite{anton-radial}). The argument from \cite{Dodson} relies on the
Strichartz estimates from \cite{oanass} and on the trace
lemma applied to the bilinear virial estimates from
\cite{plve08}, but done in a way which only involve the low
frequencies, bypassing ``directional'' requirements on the high
frequencies. In the present paper, we adapt
that argument to generic domains, irrespective of the geometry of the
boundary, obtaining a bilinear estimate which involves trace terms on
the right-hand side. Such trace terms are then disposed of, either by
local smoothing effects in the non trapping case, or by restricting to
semiclassical time scales on compact domains. Furthermore, we utilize
the machinery developed in \cite{ip08} to derive the usual bilinear
estimate from the gradient one. The combination of both estimates
immediately recovers the local well-posedness result we mentioned earlier for \eqref{nls} in
$H^s(\Omega)\cap H^1_0(\Omega)$ for $1<s\leq 2$ (\cite{anton-radial}) ; we
further develop the local theory close to $s=1$ by combining our bilinear estimate, the endpoint Strichartz estimate from \cite{BSS-ext} together with a discrete Gronwall-like estimate (inspired by \cite{Bou-gafa2})
to obtain global well-posedness results.

 The main advantage of the
present note is that it mostly relies on basic tools to obtain estimates on
linear solutions, as we do not require any sophisticated microlocal
approximation of the solution; however, we deal with the flat
Laplacian on a domain, unlike most of the aforementioned results which
deal with smooth manifolds (hence variable coefficients) and both
Dirichlet and Neumann boundary conditions. We claim
however that the present approach is mostly local in space, and may be
adapted to the variable coefficients case, at the expense of technical
tedious complications. We chose to focus on the flat case to highlight
the simplicity of the argument.

%%%%%%%%%%%%%%%%%%%%%%%%%%%%%%%%%%%%%%%%%%%%
\section{Main results}
%%%%%%%%%%%%%%%%%%%%%%%
\subsection{Bilinear estimates for the linear Schr\"odinger equation on a domain $\Omega$}
Let $\Omega\subset \R^3$ be a domain with a smooth boundary
$\partial\Omega$, and consider the Schr\"odinger equation
\begin{equation}
\label{equd}
  i\partial_t w_m+\Delta w_m=0, \text{ with }
  w_{m|\partial \Omega} =0,
\end{equation}
where $m$ stands for data
$w_m(0)$ which are spectrally localized at
$\sqrt{-\Delta}\sim 2^m$, $m\in \Z$. By this, we mean that, for some
$\varphi\in C^\infty_0(\R_+)$, $w_m=\varphi(2^{-m}\sqrt{-\Delta}) w_m$. The
smoothed out spectral projector $\varphi(\sqrt{-\Delta})$ may be defined by
spectral calculus or, more directly, by the Dynkin-Helffer-Sj\"ostrand
formula (see \cite{ip08} for details and useful properties of such
operators). Let us define, for $k\in \N^\star$ and with $\partial_n$
the normal derivative on the boundary,
\begin{equation}
  \label{eq:termebord}
\mathcal{H}_k(w_m)= \|w_m(0)\|_2\|w_m(0)\|_{\dot H_0^1(\Omega)}+ \int_0^T
  \int_{x\in \partial\Omega} \sum_{l=0}^k 2^{-2lm} |\partial_n
  \Delta^l w_m|^2 \,dS_x dt\,.
\end{equation}
\begin{theoreme}\label{t21}
Consider two solutions $u_j$ and $v_k$ to \eqref{nls}, with data $u_j(0)$ which is spectrally localized at $\sqrt{-\Delta}\sim
2^j$ and data $v_k(0)$ which is spectrally localized at $\sqrt{-\Delta}\sim
2^k$, where $k\leq j$. Then the following bilinear estimate holds:
\begin{equation}
  \label{eq:gradbilinear}
\int_0^T  \| v_k \nabla u_j \|^2_{2} \,dt \lesssim 2^{2k}\left(
\|v_k(0)\|^2_{2} \mathcal{H}_0(u_j)+\|u_j(0)\|^2_{2} \mathcal{H}_2(v_k)\right)\,.
\end{equation}
\end{theoreme}
We now provide bilinear estimates when one has global
(in time) control of the boundary term.
\begin{proposition}
\label{p1d}
Let $u_j$ and $v_k$ be as in Theorem \ref{t21} and assume moreover that $\Omega=\R^3\setminus \Sigma$, where  $\Sigma$ is a non trapping smooth compact obstacle. Then,
  \begin{equation}
    \label{eq:nlsdnormal}
\int_{\R} \|u_j v_k\|^2_2 \,dt+2^{-2j} \int_{\R}\|v_k \nabla u_j\|^2_2 \,dt
\lesssim 2^{2k-j} \|u_j(0)\|^2_{2}\|v_k(0)\|^2_{2}\,.
  \end{equation}
\end{proposition}
In \cite{Dodson}, the gradient estimate from the left-hand side of
\eqref{eq:nlsdnormal} is obtained for a strictly convex $\Sigma$
(actually the statement there deals with $\|\nabla(\bar v_k u_j)\|_2$ but as
we will see later the difference term is at a lower order).

On a bounded domain, one will have losses unless $T\lesssim 2^{-j}$:
\begin{proposition}
\label{p1dbis}
Let $u_j$ and $v_k$ be as in Theorem \ref{t21} and assume moreover that $\Omega$ is a compact domain of $\R^3$ with smooth boundary. If $T<2^{-j}$ then
  \begin{equation}
    \label{eq:nlsdnormalbis}
\int_0^T \|u_j v_k\|^2_2 \,dt+2^{-2j} \int_0^T\|v_k \nabla u_j\|^2_2 \,dt
\lesssim 2^{2k-j} \|u_j(0)\|^2_{2}\|v_k(0)\|^2_{2}\,.
  \end{equation}
Moreover, for $T<+\infty$, we have
  \begin{equation}
    \label{eq:nlsdnormalter}
\int_0^T \|u_j v_k\|^2_2 \,dt+2^{-2j} \int_0^T\|v_k \nabla u_j\|^2_2 \,dt
\lesssim T 2^{2k} \|u_j(0)\|^2_{2}\|v_k(0)\|^2_{2}\,.
  \end{equation}
\end{proposition}
%%%%%%%%%%%%%%%%%%%%%%%%%%%%%%%%%%%%%%%%%%%%%%%%%
\section{Application}
We now return to our nonlinear cubic Schr\"odinger equation
\eqref{nls}.
\begin{theoreme}\label{thcubic}
Let $\Omega$ be a bounded domain in $\R^3$ with smooth boundary. Let
$\phi_0 \in H^s_0(\Omega)$ with $1<s\leq 3$. Then the
defocusing \eqref{nls} admits a global solution $u\in C_t(H^s_0)$ which is unique
in a suitable subspace. The same result holds true in the focusing
case provided the $L^2(\Omega)$ norm of $\phi_0$ is sufficiently small
with respect to the $H^1_0(\Omega)$ norm. Moreover, the associated
flow is analytic.
\end{theoreme}
Before stating further results, we need to define several function spaces which will be of help. We start with (inhomogeneous) Besov
spaces which are built on the spectral localization (\cite{ip08} for details). Let us recall that a Littlewood-Paley
decomposition is a collection of operator $(\Delta_j)_{j\in \N}$
defined as follows: let $\phi \in \mathcal{S}(\mathbb{R})$ such that $\widehat\phi =
1$ for $|\xi|\leq 1$ and $\widehat\phi= 0$ for $|\xi|>11/10$,
$\phi_{j}(\xi)= 2^{j}\phi(2^{j}\xi)$, $ S_{j} = \phi_{j}(\sqrt{-\Delta})$,
$\psi_j(\xi)=(\phi_{j+1}-\phi_j)(\xi)$, $\Delta_{j} = S_{j+1} -
S_{j}=\psi_j (\sqrt{-\Delta})$. For notational convenience, we may sometimes
refer to $S_0$ as $\Delta_{-1}$. We shall denote  $u_j= \Delta_j u$.
\begin{definition}
\label{d1}
 Let $f$ be in $\mathcal{S}'({ \mathring \Omega})$, $s\in \R$ and $1\leq
 p,q\leq +\infty$.  We say $f$ belongs to $B^{s,q}_{p}$
 (resp. $B^{s,q}_{p,l}$) if and only if, with $\lambda_{j,s}=2^{js}$
 (resp. $\lambda_{j,s}=2^{js} \log^\frac 1 2 j$)
\begin{itemize}
\item $S_0 f\in L^p$.
\item The sequence $(\varepsilon_{j})_{j\in \N}$ with $\varepsilon_j =
  \lambda(j,s) \| \Delta_{j} (f)\|_{L^{p}}$
belongs to $l^{q}$.
\end{itemize}  
\end{definition}
We define our Sobolev spaces $H^s_0(\Omega)=B^{s,2}_2$ and remark
$H^1_0(\Omega)$ is the usual Sobolev space associated to the Dirichlet
Laplacian. We will later use its subspace $\dot B^{1,1}_{2,l}$. Notice that functions in $B^{s,2}_2$
for $s>1/2$ will satisfy trace conditions (involving only
powers of the Laplacian) as they are
embedded in the definition through the use of the spectral
projector. Finally, we will also need the conormal spaces introduced
in \cite{KT-normal,KT-UV}. In the present paper, the atomic spaces $U^p$ and $V^p$
may be seen as a black box  for which we refer to Section 2,
\cite{KH-UV}, for definitions and useful properties. In our setting, we will be using Definition
2.15 in \cite{KH-UV} with $S(t)=\exp(it\Delta)$ the linear
Schr\"odinger group associated with the Dirichlet Laplacian:
\begin{align*}
  U^p_S & = S(\cdot) U^p, \,\,\text{ with norm }\,\,
  \|u\|_{U^p_S}=\|S(-\cdot) u\|_{U^p}\\
  V^p_S & = S(\cdot) V^p, \,\,\text{ with norm }\,\,
  \|u\|_{V^p_S}=\|S(-\cdot) u\|_{V^p}
\end{align*}
\begin{definition}
  Let  $u(t,x)\in \mathcal{S}'(\R\times \Omega)$, $s\in \R$. We say
that
 $u\in X_l^{1}$ if and only if,
\begin{equation}\label{eq.defalt}
\|u\|_{X^1_l}= \sum_{j\geq 0} 2^{j} (\log j)^{\frac 1 2} \|\Delta_j u(t, \cdot)) \|_{U^2_S} <+\infty
\end{equation}
and  $u\in X^{s}$ if and only if,
\begin{equation}\label{eq.defalt2}
\|u\|^2_{X^s}= \sum_{j\geq 0}  2^{2sj}  \|\Delta_j u(t, \cdot)) \|^2_{U^2_S} <+\infty.
\end{equation}
Similarly, we say that
 $v\in Y_l^{\pm 1}$ if and only if,
\begin{equation}\label{eq.defalt3}
\|v\|^2_{Y^{\pm1}_l}= \sup_{j\geq 0} 2^{\pm j} (\log j)^{-\frac 1 2} \|\Delta_j v(t, \cdot)) \|_{V^2_S} <+\infty
\end{equation}
and  $v\in Y^{s}$ if and only if,
\begin{equation}\label{eq.defalt4}
\|v\|_{Y^s}= \sum_{j\geq 0}  2^{2sj}  \|\Delta_j v(t, \cdot)) \|^2_{V^2_S} <+\infty.
\end{equation}
\end{definition}
The most important property relating our spaces is that the Duhamel
operator associated to $S(t)$ maps the dual of $Y^{-1}_l$ (or $Y^{-s}$) to
$X^1_l$ (or $X^s$). Moreover, any multilinear estimate on products of
linear solutions to the Schr\"odinger equation may be transferred to
our functional spaces, at most at a cost of a log loss in the constants.

Now, we can reformulate a more precise local (in time) version of
Theorem \ref{thcubic}.
\begin{proposition}\label{p31}
Let $\Omega$ be a bounded domain in $\R^3$ with smooth boundary. Let
$\phi_0 \in B^{1,1}_{2,l}\subset H^1_0(\Omega)$. Then \eqref{nls} admits a
unique local in time solution $u\in X^{1}_l$. Moreover, the associated
flow is analytic, and the local time of existence $T$ is at least
comparable to $\|\phi_0\|^{-2}_{B^{1,1}_{2,l}}$.
\end{proposition}
When the data $\phi_0$ is smoother, we may refine our local result:
\begin{proposition}\label{p32}
Let $\Omega$ be a bounded domain in $\R^3$ with smooth boundary. Let
$\phi_0 \in H^s_0(\Omega)$ for $1<s\leq 3$. Then \eqref{nls} admits a
unique local in time solution $u\in X^{1}_l$. Moreover,
the solution is in $X^{s}\hookrightarrow
C_T(H^s(\Omega))$, and the local time of existence $T$ is of order
$(\log \|\phi_0\|_{H^s(\Omega)}\log \log \|\phi_0\|_{H^s(\Omega)})^{-1}$ (provided the $L^2(\Omega)$
norm of $\phi_0$ is small in the focusing case).
\end{proposition}
This last result takes into account the conservation of the
Hamiltonian, which provides control of the $H^1_0$ norm. The double
logarithm appearing in the time of existence translates into a triple
exponential growth of the corresponding Sobolev norm as a function of time.
%%%%%%%%%%%%%%%%%%%%%%%%%
\section{Proofs and further developments}
We start with proving Theorem \ref{t21}. Let us recall a result from \cite{plve08}. Let $u,v$ be two solutions to
the Schr\"odinger equation which are not necessarily spectrally
localized, and define
\begin{equation}
  \label{eq:imf}
  I_\rho(u,v)=\int \rho(x-y) |u|^2(x) |v|^2(y)\,dxdy.
\end{equation}
\begin{theoreme}[\cite{plve08}]
\label{t3}
Let $\rho$ be a weight function such that its Hessian $H_\rho$ is positive; let
\begin{equation}
  F(u,v)(x,y)=\bar v(y)\nabla_x
 u(x)+u(x)\nabla_y \bar v(y)\,.
\end{equation}
 We have
  \begin{multline}
\label{nonlinearmauvais}
 \partial_t^2 I_\rho  =  4\int H_\rho(x-y)(F(u,v)(x,y),\overline F(u,v)(x,y)) \,dxdy\\
   {} -2 \int_{x\in
  \partial\Omega, y\in\Omega} |v|^2(y) \partial_n \rho(x-y)
  |\partial_n u|^2(x) \, dS_x dy\\
{}-2 \int_{y\in
  \partial\Omega, x\in\Omega} |u|^2(x) \partial_n \rho(x-y)
  |\partial_n v|^2(y) \, dS_y dx.
\end{multline}
\end{theoreme}
We follow \cite{Dodson} in spirit, although our presentation will
differ, both on the choice of weight and our later treatment of lower
order terms. First, rather than applying our theorem with the weight
$\rho_\omega(x-y)=|\omega\cdot (x-y)+\sigma)|$ where $\omega \in \S^2$ is
any direction and $\sigma\in \R$ is a spatial parameter which will be
averaged later, we directly use the weight $\rho_{\omega,k}$ defined as
follows:
\begin{align*}
  \rho_{\omega,k}(z) & = |\omega\cdot z| \,\,\text{ if }\,\, |\omega\cdot
  z|> 2^{-k}\,,\\
  \rho_{\omega,k}(z) & = 2^{k} |\omega\cdot z|^2/2+2^{-k}/2 \,\,\text{ if }\,\, |\omega\cdot
  z|\leq 2^{-k}\,,
\end{align*}
This yields, 
\begin{multline}
\label{idt1}
\int_{|(x-y)\cdot \omega|< 2^{-k}} |u_j (x)(\omega.\nabla){\bar
  v_k}(y)+{\bar v_k}(y)(\omega\cdot\nabla) u_j(x)|^2 \,dx dy \\
 -\int_{x\in
  \partial\Omega, y\in\Omega} |v_k|^2(y) \partial_n \rho_{\omega,k}(x-y)
  |\partial_n u_j|^2(x) \, dS_x dy \\
{}-\int_{y\in   \partial\Omega, x\in\Omega} |u_j|^2(x) \partial_n \rho_{\omega,k}(x-y)
  |\partial_n v_k|^2(x) \, dS_y dx = \frac 1 4\partial^2_t I_{\rho_\omega,k}.
\end{multline}
Let us deal with the right-hand side: after time integration and with
$$
\omega\cdot \nabla \rho_{\omega,k}(z)=
\mathbf{1}_{|z\cdot\omega|>2^{-k}} \mathrm{sgn(z\cdot
  \omega)}+\mathbf{1}_{|z\cdot\omega|\leq 2^{-k}} 2^k |z\cdot \omega|
$$
which we notice is a bounded function, we get
\begin{multline*}
\partial_t I_{\rho_{\omega,k}} (t)= \int \omega\cdot \nabla \rho_{\omega,k}(x-y)
\Im \bar u_j (\omega\cdot \nabla) u_j(x) |v_k|^2(y)\,dxdy \\ 
{}-\int \omega\cdot \nabla \rho_{\omega,k}(x-y)  \Im \bar v_k (\omega\cdot \nabla) v_k(y)
|u_j|^2(x)\,dxdy 
\end{multline*}
which is easily seen to be bounded,
\begin{equation*}
|\partial_t I_{\rho_{\omega,k}} (T)-\partial_t I_{\rho_{\omega,k}} (0)|\lesssim \|v_k(0)\|^2_{2}\|u_j(0)\|_2 \|u_j(0)\|_{\dot H_0^1(\Omega)}+\|u_j(0)\|^2_{2}\|v(0)\|_2  \|v_k(0)\|_{\dot H_0^1(\Omega)}.
\end{equation*}
On the other hand, $|\partial_n \rho_{\omega,k}|\leq 1$ and the boundary
terms in \eqref{idt1} may be bounded by
\begin{multline}
  \label{eq:boundaryterms}
   \int_{\partial\Omega\times\Omega} |v_k|^2(y) |\partial_n \rho_\omega(x-y)|
  |\partial_n u_j|^2(x) \, dS_x dy\\ {}+\int_{\Omega\times \partial\Omega} |u_j|^2(x) |\partial_n \rho_\omega(x-y)|
  |\partial_n v_k|^2(x) \, dS_y dx \lesssim \mathcal{H}_0(u_j)\|v_k(0)\|^2_2+\mathcal{H}_0(v_k)\|u_j(0)\|^2_2\,.
\end{multline}
We are thus
left with
\begin{multline}\label{split}
\int_0^T  \int_{|(x-y) \cdot \omega|<2^{-k}} |u_j(x)(\omega\cdot \nabla)
{\bar v_k}(y)+{\bar v_k}(y)(\omega\cdot \nabla) u_j(x)|^2 \,dx dy dt\\
\lesssim \|v_k(0)\|^2_{2}\mathcal{H}_0(u_j)+\|u_j(0)\|^2_{2}\mathcal{H}_0(v_k)\,.
\end{multline}
We now consider the term $u_j (\omega \cdot \nabla) \bar v_k$ in the $|\cdots|^2$
inside \eqref{split} above, and  restrict to $|(x-y)^\perp|<2^{-k}$,
where for any $z$, $z^\perp=z-(z\cdot\omega)\omega$, to get
$$
\int_0^T \int_{|(x-y)\cdot \omega|<2^{-k}}\int_{|(x-y)^\perp|<2^{-k}}
|u_j(x)(\omega\cdot \nabla){\bar v_k}(y)|^2
\,dx dy dt\,;
$$
changing variables so that $y=x+z$, we may bound this integral from
above by
$$
J_\omega =\int_0^T \int_{|z|< 2^{-(k-1)}} |u_j(x)(\omega\cdot \nabla) v_k(x+z)|^2 \,dxdzdt\,,
$$
as $\{|(x-y)\cdot \omega|<2^{-k}\}\cap \{|(x-y)^\perp|<2^{-k}\}\subset
\{|x-y|<{2^{-(k-1)}}\}$.
Notice that the variable $z$ is averaged over a ball of size ${2^{-(k-1)}}$,
which is the dual size of the spectral localization at $2^k$
of $v_k$. This will be crucial later on. For now, by Cauchy-Schwarz,
$$
J^2_\omega\lesssim \int_0^T \|u_j\|^4_4 \,dx dt \int_0^T \|\nabla
v_k\|^4_4 \,dx dt \int_{|z|< {2^{-(k-1)}}} dz\lesssim 2^{-3k} \int_0^T \|u_j\|^4_4 \,dx dt \int_0^T \|\sqrt{-\Delta}
v_k\|^4_4 \,dx dt \,.
$$
We then use the linear $L^4_{t,x}$ estimate from \cite{plve08}, which
becomes with our notations
\begin{lemme}[\cite{plve08}]
\label{lemmelineaire} Let $w_m$ be a solution to \eqref{equd}. Then
\begin{equation}
  \label{eq:L4lin}
  \int_0^T \|w_m\|^4_4 \,dt \lesssim \|w_m(0)\|^2_2
\mathcal{H}_{0}(w_m).
\end{equation}
Moreover, when $\Omega$ is compact,
\begin{equation}
  \label{eq:L4lincomp}
  \int_0^T \|w_m\|^4_4 \,dt \lesssim T 2^{2m}\|w_m(0)\|^4_2.
\end{equation}
\end{lemme}
Therefore, we have $
\int_0^T \|u_j\|^4_4 \,dt \lesssim \|u_j(0)\|^2_2
\mathcal{H}_{0}(u_j)\,\,\text{ and }\,\,\int_0^T \|v_k\|^4_4 \,dt \lesssim \|v_k(0)\|^2_2
\mathcal{H}_{0}(v_k),
$
and get
\begin{equation}
\label{Jomega}
  J_\omega \lesssim \|u_j(0)\|_2  \mathcal{H}_{0}(u_j)
\times 2^{-3k+2k}\| v_k(0)\|_2  \mathcal{H}_{0}(v_k)\,.
\end{equation}
Going back to \eqref{split}, we may restrict the space integration to
$|x-y|<2^{-k}$ as well as change variables $y=x+z$ and consequently we are left with

$$
K_\omega (u,v) =\int_0^T \int_{|z|<2^{-k}} |v_k(x+z)(\omega\cdot \nabla) u_j(x) |^2 \,dxdzdt
$$
for which we proved (combining \eqref{split} and \eqref{Jomega}) 
\begin{equation}
  \label{eq:intermede}
  K_\omega (u,v) \lesssim 2^{-k} \left( \|v_k(0)\|^2_2  \mathcal{H}_{0}(u_j)|+\| u_j(0)\|_2  \mathcal{H}_{0}(v_k)\right)\,.
\end{equation}
We digress with an elementary version of the trace lemma:
\begin{lemme}
  Let $\phi$ be a smooth function in $\R^3$ and $\lambda>0$. Denote by
  $C_\mu$ the cube centered at $x=0$ with size $\mu$. Then
  \begin{equation}
    \label{eq:tracelem}
    |\phi(0)|^2\lesssim \lambda^{-1} \int_{C_{\lambda^{-1}}}
    |\Delta \phi|^2 +\lambda^3 \int_{C_{\lambda^{-1}}} |\phi|^2\,.
  \end{equation}
\end{lemme}
By standard elliptic regularity, \eqref{eq:tracelem} is true for a
cube of size $\lambda=1$. The estimate for any $\lambda$ then follows
by rescaling, applying the estimate for $\lambda=1$ to $\phi(\lambda^{-1}
\cdot)$.\cqfd

We now apply the lemma to $v_k(x+z)$ as a function of $z$, with
$\lambda=  2^{k-1}$:
$$
|v_k(x)|^2\lesssim 2^{-k} \int_{|z|<  2^{-k+1}}
    |\Delta_z v_k(x+z)|^2\,dz +2^{3k} \int_{|z|<  2^{-k+1}}
    |v_k(x+z)|^2\,dz\,.
$$
Taking advantage of $\Delta_z v_k(x+z)=\Delta_x v_k(x+z)$, we can combine
\eqref{eq:intermede} for $v_k$ with \eqref{eq:intermede} where $v_k$
is replaced by $\Delta v_k$ which is also a solution to the
Schr\"odinger equation with Dirichlet boundary condition:
\begin{align*}
  \int_0^T\int_\Omega |v_k(x)(\omega\cdot \nabla)u_j(x)|^2 dx dt  \lesssim  & 2^{-k}
  K_\omega(u,\Delta v_k)+  2^{3k} K_\omega(u,v_k) \\
  \lesssim & 2^{-2k} \left( \|\Delta v_k(0)\|^2_2  \mathcal{H}_{0}(u_j)+\|
   u_j(0)\|^2_2  \mathcal{H}_{0}(\Delta v_k)\right)\\
 & + 2^{2k} \left( \| v_k(0)\|^2_2  \mathcal{H}_{0}(u_j)+\|
   u_j(0)\|^2_2  \mathcal{H}_{0}( v_k)\right)\\
 \lesssim &  2^{2k} \left( \| v_k(0)\|^2_2  \mathcal{H}_{0}(u_j)+\|
   u_j(0)\|^2_2  \mathcal{H}_{2}( v_k)\right)\,,
\end{align*}
Which is the desired conclusion, as $\omega$ is any direction in
$\S^2$. Theorem \ref{t21} is proved.\cqfd

Notice that
\begin{equation}
  \int_0^T \| u_j \nabla v_k \|^2_2 \,dt \lesssim\left( \int_0^T \| u_j \|^4_4
\,dt \right)^\frac 1 2  \left( \int_0^T \| \sqrt{-\Delta} v_k \|^4_4
\,dt \right)^\frac 1 2 \lesssim \|u_j(0)\|_2\mathcal{H}_{0}(u_j)^\frac 1 2\|v_k(0)\|_2\mathcal{H}_{0}(v_k)^\frac 1 2
\end{equation}
so that we also control $\nabla( u_j v_k)$ or $\nabla(u_j \bar v_k)$ rather than just $v_k \nabla
u_j$. This will prove to be useful in the next argument.

We now prove \eqref{eq:nlsdnormal} from Proposition \ref{p1d} and
\eqref{eq:nlsdnormalbis} from Proposition \ref{p1dbis}: let us sumarize our current result as
\begin{equation}\label{interder}
 2^{-2k-j} \int_0^T \| u_j \nabla v_k \|^2_2+\|\nabla(u_j v_k)\|^2_2 \,dt \lesssim
  2^{-k} \mathcal{H}_2(v_k)\|u_j(0)\|_2^2+2^{-j}
    \mathcal{H}_0(u_j)\|v_k(0)\|^2_2=\Gamma(u_j,v_k) \,.
\end{equation}
Let $\Delta_l=\varphi(2^{-2l} \Delta)$ be a spectral
localization on $\Omega$. Let $Q_l=2^{-l}\nabla \exp(2^{-2l}
\Delta)$and $Q^\star_l$ its transpose. We need to prove
$$
\int_0^T \| u_j v_k \|^2_2\,dt \approx \sum_l\int_0^T \| \Delta_l (u_j
v_k)\|^2_2\,dt\lesssim 2^{2k-j} \|u_j(0)\|^2_2 \|v_k(0)\|_2^2\,,
$$
where $T$ is arbitrarily large on a non-trapping domain or $T<2^{-j}$
on a bounded domain.

Start with $l>j$: we have
\begin{equation*}
  \Delta_l(u_j v_k)  = \tilde\Delta_l \exp(2^{-2l}\Delta) 2^{-2l}\Delta(u_j v_k) = \tilde \Delta_l Q^\star_{l} 2^{-l} \nabla(u_j v_k)
\end{equation*}
where $\tilde \Delta_l =\psi(2^{-2l}\Delta)$ and
$\psi(\xi)=\varphi(\xi)\xi^{-1} \exp(\xi) \in C^\infty_0(\R^*_+)$, and
therefore, as $\tilde \Delta_j$, $Q^\star_l$ are bounded on $L^2$ (see \cite{ip08})
$$
\int_0^T \|\Delta_l(u_j v_k)\|^2_2\,dt \lesssim 2^{-2l+j+2k}\Gamma(u_j,v_k),
$$
which we may sum over $l>j$ to get  that
\begin{equation}
  \label{eq:hautefreq}
  \sum_{l>j} \int_0^T \|\Delta_l(u_j v_k)\|^2_2\,dt \lesssim
2^{-j+2k}\Gamma(u_j,v_k).
\end{equation}
We now proceed with $l<j$. We consider directly
$$
\sum_{l<j} \Delta_l (u_j v_k)=S_j(u_j v_k).
$$
Recall $u_j=\phi(2^{-2j} \Delta) u_j$ where $\phi$ is compactly
supported away from $\xi=0$. Let $\tilde \phi\in C^\infty_0(\R_+^*)$ be such that $\xi^{-1} \tilde
\phi(\xi)=1$ on the support of $\phi$, and $\mathcal{D}_j =\tilde \phi(2^{-2j}\Delta)$. We may rewrite
$$
S_j(u_jv_k)=S_j (v_k 2^{-2j}\Delta \mathcal{D}_j u_j)\,,
$$
and $w_j=\mathcal{D}_j u_j$ is a solution to \eqref{equd} with data
$w_j(0)=\mathcal{D}_j u_j(0)$. Then
\begin{align*}
  S_j(u_jv_k) & = S_j(2^{-j}\mathrm{div} (v_k 2^{-j}\nabla w_j) - S_j(2^{-j}\nabla v_k\cdot 2^{-j}\nabla w_j)\\ 
   & = \tilde S_j Q_j^\star (v_k 2^{-j}\nabla w_j) - 2^{k-j} S_j(2^{-k}\nabla v_k\cdot 2^{-j}\nabla w_j)\,,
\end{align*}
where $\tilde S_j=\gamma(2^{-2j}\Delta)$ and
$\gamma(\xi)=\sum_{l<0}\varphi(2^{-2l}\xi) \exp(\xi)\in
C^\infty_0([0,+\infty))$. Again, $S_j$, $\tilde S_j$, $Q^\star_j$ are bounded
on $L^2$ and therefore, using \eqref{interder} on $v_k\nabla w_j$ and
Lemma \ref{lemmelineaire} for $\nabla v_k$ and $\nabla w_j$,
\begin{align*}
  \label{eq:yada}
\int_0^T \|  S_j(u_jv_k)\|^2_2 \,dt & \lesssim 2^{-j} 2^{2k+j} \Gamma(w_j,v_k)+2^{2(k-j)}
  \|w_j(0)\|_2\mathcal{H}_0(w_j)^\frac 1 2\|v_k(0)\|_2\mathcal{H}_0(v_k)^\frac 1 2\\
  & \lesssim 2^{2k-j} \Gamma(w_j,v_k)+2^{2(k-j)}2^{j/2+k/2}\Gamma(w_j,v_k)\,,
\end{align*}
and using $k<j$ and continuity of $\mathcal{D}_j$,
\begin{equation}
  \label{eq:presquefini}
\int_0^T\|  S_j(u_jv_k)\|^2_2\,dt \lesssim 2^{2k-j}  \Gamma(w_j,v_k)\,.
\end{equation}
In order to complete the proof, we proceed differently depending on
the domain. Let $S$ be a (compact) strip close to the boundary
$\partial\Omega$, we recall the following estimate from \cite{plve08}
(eq. (5.4) p. 279), on solutions to \eqref{nls} with $\varepsilon=0$
(linear equation):
\begin{equation}
  \label{eq:controlebord}
  \int_0^T \int_{\partial\Omega} |\partial_n \phi|^2 dSdt \lesssim
  \int_0^T \|\phi\|^2_{H^1(S)} dt+\sup_{[0,T]}\|\phi\|^2_{\dot
      H^{1/2}}.
\end{equation}
We are now facing two different situations:
\begin{itemize}
\item if $\Omega$ is the exterior of a non-trapping compact obstacle,
  local smoothing holds (\cite{burqduke}) and the time integral on the
  right-hand side of \eqref{eq:controlebord} is controlled by the other
  term, itself controlled by invariants of the flow. Hence,
  \begin{equation}
    \label{eq:bordext}
  \int_0^T \int_{\partial\Omega} |\partial_n \phi|^2 dSdt \lesssim
  \|\phi\|_{2} \|\phi\|_{H^1_0(\Omega)}\,.
  \end{equation}
Recalling the definition of $\Gamma(\cdot,\cdot)$,
$\mathcal{H}_k(\cdot)$ and using the frequencies localizations, we get
$$
\Gamma(w_j,v_k)\lesssim \|w_j(0)\|_2 \|v_k(0)\|_2 \lesssim \|u_j(0)\|_2
\|v_k(0)\|_2\,.
$$
Gathering \eqref{eq:hautefreq}, \eqref{eq:presquefini} we obtain the
desired estimate \eqref{eq:nlsdnormal};
\item if $\Omega$ is a compact domain, we simply restrict $T<2^{-j}$,
  so that
  \begin{equation}
    \label{eq:bordextbis}
  \int_0^T \int_{\partial\Omega} |\partial_n \phi|^2 dSdt \lesssim
 2^{-j} \|\phi(0)\|^2_{H^1_0(\Omega)}+\|\phi(0)\|_2 \|\phi(0)\|_{H^1_0(\Omega)}\,.
  \end{equation}
Now, we get again
$$
\Gamma(w_j,v_k)\lesssim \|w_j(0)\|_2 \|v_k(0)\|_2 \lesssim \|u_j(0)\|_2
\|v_k(0)\|_2\,,
$$
but only for $T<2^{-j}$. Gathering \eqref{eq:hautefreq}, \eqref{eq:presquefini} we obtain the
desired estimate \eqref{eq:nlsdnormalbis}; for any interval $[0,T]$,
we split into intervals of size $2^{-j}$ and sum up, which yields \eqref{eq:nlsdnormalter}.
\end{itemize}
This achieves the proof of Propositions \ref{p1d} and \ref{p1dbis}. \cqfd

\subsection{Nonlinear equation on a domain}

We shall use the following result to estimate norms in our conormal
spaces $X^{1}_{(l)}$ and $X^s$, which can be extracted from \cite{KH-UV}.
\begin{lemme}\label{lem.substit}
Consider $u$ the solution of $\, i \partial_t + \Delta u=f,
\, u_{\mid \partial \Omega}=0 \, u_{ \mid t=0}=u_0$. Then for any $s$,
$$
 \|u\|_{X^{s}}\leq C_T (\|u_0\|_{B^{s,2}_2}+
 \|f\|_{(Y^{-s})'}),
$$
where $(Y^{-s})'$ is the dual space (with respect to the $L^2_{t,x}$ bracket) of $Y^s$.
\end{lemme}
\begin{rem}
  It should be emphasized that, for small $T$, $C_T\sim C$, which is a
consequence of our use of the $U$ and $V$ spaces. Hence, in a
contraction argument, the smallness may not be provided by the Duhamel
estimate (as is customary with the classical $X^{s,b}$ spaces), nor by
rescaling since we are on a domain. However, the nonlinear estimate
will be derived using \eqref{eq:nlsdnormalter} where the $T$ factor
will serve this purpose.
\end{rem}

The important feature of the $U$ and $V$ based spaces is that we have
a good transference principle: multilinear estimates involving linear
solutions to the Schr\"odinger equation may be turned into estimates
for our $X^s_{(l)}$ spaces. From now on, it should be understood that
we restrict the time interval to a fixed time interval $[-T,T]$. Given
that functions in our spaces can be time-truncated this is harmless.
\begin{proposition}\label{p41}
  Let $u^{(1)},u^{(2)},u^{(3)} \in X^{1}_l$ then $u^{(1)} \bar u^{(2)} u^{(3)} \in
  (Y^{-1}_l)'$ and
  \begin{equation}
    \label{eq:nonlinearX}
     \| u^{(1)} \bar u^{(2)} u^{(3)}  \|_{(Y^{-1}_l)' } \lesssim T  \| u^{(1)} \|_{X^{1}_l} \|  u^{(2)} \|_{X_l^{1}} \|  u^{(3)} \|_{X_l^{1}}\,.
  \end{equation}
Moreover, if  $u \in X^{s}$, for $1<s\leq 3$, then
  \begin{equation}
    \label{eq:nonlinearXpersist}
     \| |u|^2 u \|_{(Y^{-s})'} \lesssim T  \| u \|^2_{X_l^{1}} \|  u \|_{X^{s}}\,.
  \end{equation}
\end{proposition}
Both nonlinear mappings are proved by duality: for
convenience, introduce $v\in Y^1_l$, then \eqref{eq:nonlinearX} is a
consequence of
$$
n\left|\int_0^T\int_\Omega u_1 \bar u_2 u_3 \Delta \bar v \,dxdt
\right|\lesssim T   \| u^{(1)} \|_{X^{1}_l} \|  u^{(2)} \|_{X_l^{1}} \|  u^{(3)}
\|_{X_l^{1}} \| v\|_{Y^1_l}\,,
$$
while, with $v\in Y^{2-s}$, \eqref{eq:nonlinearXpersist} is a consequence of 
$$
\left|\int_0^T\int_\Omega u \bar u u \Delta \bar v \,dxdt
\right|\lesssim T   \| u \|^2_{X^{1}_l} \|  u \|_{X^{s}} \| v\|_{Y^{2-s}}\,.
$$
Both space-time integrals may now be decomposed into dyadic pieces,
and we are left with estimating
\begin{equation}
  \label{eq:I1234}
  I_{j_1,j_2,j_3,j_4}=\int_0^T 
\int_\Omega  u^{(1)}_{j_1} \bar u^{(2)}_{j_2} u^{(3)}_{j_3} \Delta v_{j_4} \,dx dt \,
\end{equation}
where the first three factors are dyadic pieces of the corresponding $u^{(1,2,3)}$ and the
fourth one of $v$. In order to proceed, we recall the following
end-point result, which will be crucial:
\begin{proposition}[\cite{BSS-ext}]
  The solution $w_m$ to the linear Schr\"odinger equation \eqref{equd}
  satisfies
  \begin{equation}
    \label{eq:loglogloss}
    \| w_m\|_{L^2_T L^\infty(\Omega)} \lesssim \sqrt T m^2 2^m \| w_m(0)\|_2.
  \end{equation}
\end{proposition}
This is nothing but a summation over time intervals of \cite{BSS-ext}, Lemma 6.1
(which is stated on an interval of size $2^{-m}$), together with
conservation of mass. Notice that \eqref{eq:loglogloss} has a double
logarithmic loss in the frequency. We may now use Proposition 2.19 (i)
in \cite{KH-UV} to
obtain, for a generic $w_m$ which is frequency localized at $2^m$,
  \begin{equation}
    \label{eq:logloglossU}
    \| w_m\|_{L^2_T L^\infty(\Omega)} \lesssim \sqrt T m^2 2^m \| w_m\|_{U^2_S},
  \end{equation}
while we have the obvious ``energy'' estimate $\|w_m\|_{L^\infty_t
  L^2} \lesssim \|w_m\|_{U^q_S}$ for any $q\leq +\infty$. We may now
state our key estimate as a suitable analog of Proposition
\ref{p1dbis}.
\begin{proposition}
  Let $u_j$ and $v_k$ be frequency localized functions with $k\leq j$,
  then
  \begin{equation}
    \label{eq:UU}
\int_0^T \|u_j v_k\|^2_2 \,dt+2^{-2j} \int_0^T\|v_k \nabla u_j\|^2_2 \,dt
\lesssim T 2^{2k} \|u_j\|^2_{U^2_S}\|v_k(0)\|^2_{U^2_S}\,,
  \end{equation}
and
  \begin{equation}
    \label{eq:UV}
\int_0^T \|u_j v_k\|^2_2 \,dt+2^{-2j} \int_0^T\|v_k \nabla u_j\|^2_2 \,dt
\lesssim T 2^{2k} (\log k)^2  \|u_j\|^2_{V^2_S}\|v_k(0)\|^2_{U^2_S}\,.
  \end{equation}
Moreover,
  \begin{equation}
    \label{eq:VV}
2^{-2k-2j} \int_0^T\|u_j \Delta v_k\|^2_2 +2^{-2j} \int_0^T\|u_j \nabla v_k\|^2_2 \,dt\lesssim T 2^{ 3 k-j}  \|u_j\|^2_{V^2_S}\|v_k(0)\|^2_{V^2_S}\,.
  \end{equation}
\end{proposition}
\begin{proof}
  The first estimate \eqref{eq:UU} is a direct consequence of
  \eqref{eq:nlsdnormalter} and, again, Proposition 2.19 (i) in \cite{KH-UV},
  which is the transference principle we alluded to.
The third estimate \eqref{eq:VV} is a consequence of a linear estimate: again by Proposition
2.19 in \cite{KH-UV}, we have, as a consequence of
\eqref{eq:L4lincomp}, that for any $w_m$ localized at $2^m$, $l\in \mathbb{N}$,
\begin{equation}
  \label{eq:L4lincompV}
  \|(\nabla)^l w_{m}\|_{L^4_T L^4} \lesssim T^\frac 1 4
  2^{(\frac 1 2+l)m} \|w_m\|_{U^4_S}\lesssim  T^\frac 1 4
  2^{(\frac 1 2+l)m} \|w_m\|_{V^2_S},
\end{equation}
using $U^4_S\subset V^2_S$. Hence \eqref{eq:VV} follows by
H\"older.   The
  second (key !) estimate \eqref{eq:UV} is proven in two steps: first, combining by H\"older
  \eqref{eq:logloglossU} on $v_k$ and the trivial ``energy'' estimate
  on $u_j$, we get a
  new bilinear estimate,
  \begin{equation}
    \label{eq:Ulog}
2^{-2j} \int_0^T\|v_k \nabla u_j\|^2_2 \,dt+\int_0^T \|u_j v_k\|^2_2 \,dt
\lesssim T 2^{2k} k^4 \|u_j\|^2_{U^{\infty}_S}\|v_k(0)\|^2_{U^2_S}\,.
  \end{equation}
We may now apply Proposition 2.20 from \cite{KH-UV} to the operator $\mathcal{T}$
mapping $u_j $ to $u_j v_k$ : by \eqref{eq:Ulog} it maps $U^\infty_S$
to $L^2_T L^2$ with constant $ C_1 \sqrt T 2^{k} k^2
\|v_k(0)\|_{U^2_S}$, while by \eqref{eq:UU} it maps  $U^2_S$
to $L^2_T L^2$ with constant $ C_2 \sqrt T 2^{k}
\|v_k(0)\|_{U^2_S}$. As a result, $\mathcal{T}$ maps $V^2_S$ to $L^2_T
L^2$ with constant $ C_3 \sqrt T 2^{k} \log k \|v_k(0)\|_{U^2_S}$ which is the
desired result. The estimate for $v_k \nabla u_j$ follows similarly.
\end{proof}
We remark that the right-hand side of \eqref{eq:VV} may
obviously be replaced by the right-hand side of \eqref{eq:UV}, as
$U^2_S\subset V^2_S$. This will be useful in unifying the treatment of
nonlinear terms later. 
\begin{lemme}\label{l44}
  Let $j_1,j_2,j_4 \leq j_3$, we have
  \begin{equation}
    \label{eq:dyadicmultilinear}
    \sum_{j_1,j_2} |I_{j_1,j_2,j_3,j_4}| \lesssim T
    \|u^{(1)}\|_{X^{1}_l} \| u^{(2)}\|_{X^{1}_l}
    \|u^{(3)}_{j_3}\|_{U^2_S} \|\Delta v_{j_4}\|_{V^2_S}\,.
  \end{equation}
\end{lemme}
\begin{proof}
One should think that the most difficult case is when $j_1,j_2\leq
j_4$: Estimate \eqref{eq:dyadicmultilinear} then follows from
\eqref{eq:UU} and \eqref{eq:UV} after pairing both
low frequency factors with the higher frequencies factors, applying
Cauchy-Schwarz, summing over $j_1,j_2$ and using the definition
of $X^{1}_l$: without loss of generality, let $j_1\leq j_2$,
\begin{align*}
   |I_{j_1,j_2,j_3,j_4}| & \lesssim \|u^{(2)}_{j_2} u^{(3)}_{j_3}\|_{L^2_T L^2}
   \|u^{(1)}_{j_1} \Delta v_{j_4}\|_{L^2_T L^2}\\
 & \lesssim T \log j_1 2^{j_1} \|u^{(1)}_{j_1} \|_{U^2_S}  2^{j_2} \|
   u^{(2)}_{j_2} \|_{U^2_S } \|u^{(3)}_{j_3}\|_{U^2_S}
   \|\Delta v_{j_4}\|_{V^2_S}\\
\sum_{j_1\leq j_2}   |I_{j_1,j_2,j_3,j_4}| & \lesssim \sum_{j_1\leq
  j_2} T (\log j_1)^\frac 1 2 2^{j_1} \|u^{(1)}_{j_1} \|_{U^2_S}
(\log j_2)^\frac 1 2 2^{j_2} \|
   u^{(2)}_{j_2} \|_{U^2_S } \|u^{(3)}_{j_3}\|_{U^2_S}
   \|\Delta v_{j_4}\|_{V^2_S}
\end{align*}
which implies the desired result. When $j_4$ is less than $j_1$ and/or
$j_2$, we simply use \eqref{eq:VV} on the $v$ factor and one of the $u$'s.
\end{proof}

Assume that $\epsilon=0,1/2$, $s>0$,$1\leq p$ together with
$$
2^{s j_3} (\log j_3)^\epsilon \|u^{(3)}_{j_3}\|_{U^2_S}=\alpha_{j_3} \in l^p
\,\,\text{ and }\,\,2^{(2-s)j_4} (\log j_4)^{-\epsilon}
\|v_{j_4}\|_{U^2_S}=\beta_{j_4} \in l^{p'}
$$
then we may sum over $j_3\leq j_4$ (using Young for discrete
sequences) to get
$$
\sum_{j_1\leq j_2\leq j_4\leq j_3} |I_{j_1,j_2,j_3,j_4}| \lesssim 
    \|u^{(1)}\|_{X^{1}_l} \| u^{(2)}\|_{X^{1}_l}
    \|(\alpha_{j_3})_{j_3}\|_{l^p} \|(\beta_{j_4})_{j_4}\|_{l^{p'}}.
$$
Both nonlinear mappings from Proposition \ref{p41} follow for such a
restricted set of indices $\{j_1,j_2,j_4\leq j_3\}$, choosing $s=1$, $p=1$,
$\epsilon=1/2$ and then $1<s\leq 3$, $p=2$, $\epsilon=0$. 

We now deal with the most difficult situation, namely $j_1,j_2,j_3\leq
j_4$. Without loss of generality, we may suppose $j_1\leq j_2\leq j_3$.  Integrating by parts the full Laplacian,
we get terms like
\begin{equation}
  \label{eq:J1234}
  J_{j_1,j_2,j_3,j_4}=\int_0^T 
\int_\Omega  u^{(1)}_{j_1} \bar u^{(2)}_{j_2} \Delta u^{(3)}_{j_3}
v_{j_4} \,dx dt \,\text{ and } \,   K_{j_1,j_2,j_3,j_4}=\int_0^T 
\int_\Omega  u^{(1)}_{j_1} \nabla \bar u^{(2)}_{j_2} \cdot \nabla  u^{(3)}_{j_3}
v_{j_4} \,dx dt
\end{equation}
where successive boundary terms vanish due to the Dirichlet boundary
condition and where we kept only the most difficult terms (derivatives
fall on the highest frequencies).
\begin{lemme}\label{l45}
  Let $j_1,j_2\leq j_3 \leq j_4$, we have
  \begin{equation}
    \label{eq:dyadicmultilinearJ}
    \sum_{j_1,j_2} |J_{j_1,j_2,j_3,j_4}| \lesssim T
    \|u^{(1)}\|_{X^{1}_l} \| u^{(2)}\|_{X^{1}_l}
    \|\Delta u^{(3)}_{j_3}\|_{U^2_S} \|v_{j_4}\|_{V^2_S}\,,
  \end{equation}
while
  \begin{equation}
    \label{eq:dyadicmultilinearK}
    \sum_{j_1,j_2\leq j_3} |K_{j_1,j_2,j_3,j_4}| \lesssim T
    \|u^{(1)}\|_{X^{1}_l} \| u^{(2)}\|_{X^{1}_l}
    2^{\frac 3 2 j_3} \|\Delta u^{(3)}_{j_3}\|_{U^2_S} 2^{\frac 1 2 j_4} \|v_{j_4}\|_{V^2_S}\,.
  \end{equation}
\end{lemme}
\begin{proof}
 The $J_{j_1,j_2,j_3,j_4}$ integral
may be dealt with like we did for $I_{j_1,j_2,j_3,j_4}$: indeed,
\begin{align*}
   |J_{j_1,j_2,j_3,j_4}| & \lesssim \|u^{(2)}_{j_2} \Delta u^{(3)}_{j_3}\|_{L^2_T L^2}
   \|u^{(1)}_{j_1} v_{j_4}\|_{L^2_T L^2}\\
\sum_{j_1\leq j_2}   |I_{j_1,j_2,j_3,j_4}| & \lesssim \sum_{j_1\leq
  j_2} T (\log j_1)^\frac 1 2 2^{j_1} \|u^{(1)}_{j_1} \|_{U^2_S}
(\log j_2)^\frac 1 2 2^{j_2} \|
   u^{(2)}_{j_2} \|_{U^2_S } \|\Delta u^{(3)}_{j_3}\|_{U^2_S}
   \| v_{j_4}\|_{V^2_S}\\
\sum_{j_1\leq j_2} |J_{j_1,j_2,j_3,j_4}|&  \lesssim 
    \|u^{(1)}\|_{X^{1}_l} \| u^{(2)}\|_{X^{1}_l} \|\Delta u^{(3)}_{j_3}\|_{U^2_S}
   \| v_{j_4}\|_{V^2_S}
\end{align*}
The $K_{j_1,j_2,j_3,j_4}$ is
interesting in that it does not require the use of \eqref{eq:UV} but
only \eqref{eq:UU} and the linear $L^4_{tx}$ estimate \eqref{eq:L4lincompV}:
\begin{align*}
   |K_{j_1,j_2,j_3,j_4}| & \lesssim \|u^{(1)}_{j_1} \nabla u^{(3)}_{j_3}\|_{L^2_T L^2}
   \|\nabla u^{(2)}_{j_2}\|_{L^4_T L^4} \| v_{j_4}\|_{L^4_T L^4}\\
\sum_{j_1\leq j_2}   |I_{j_1,j_2,j_3,j_4}| & \lesssim \sum_{j_1\leq
  j_2} T 2^{j_1} \|u^{(1)}_{j_1} \|_{U^2_S}
 2^{\frac 3 2 j_2} \|
   u^{(2)}_{j_2} \|_{U^4_S } 2^{j_3} \|u^{(3)}_{j_3}\|_{U^2_S}
  2^{\frac 1 2 j_4} \| v_{j_4}\|_{V^2_S}\\
\sum_{j_1\leq j_2\leq j_3} |J_{j_1,j_2,j_3,j_4}|&  \lesssim 
    \|u^{(1)}\|_{X^{1}_l} \| u^{(2)}\|_{X^{1}_l}   2^{\frac 3 2j_3} \|\Delta u^{(3)}_{j_3}\|_{U^2_S}
  2^{\frac 1 2 j_4} \| v_{j_4}\|_{V^2_S}
\end{align*}
where we notice that the $\log$ factor in $X^1_l$ could be dispensed with.
\end{proof}
Now, we may again conclude as we did before for both remaining sums
over $j_3\leq j_4$, in order to achieve the proof of Proposition
\ref{p41}, for $1<s<2$.

However, for $2\leq s\leq 3$, we need to integrate by parts again: in
fact, $v\in Y^{2-s}$ is still at zero or negative regularity, so we set $v=\Delta w$ and
we consider
\begin{equation}
  \label{eq:K1234}
  \tilde J_{j_1,j_2,j_3,j_4}=\int_0^T 
\int_\Omega  u^{(1)}_{j_1} \bar u^{(2)}_{j_2} \Delta u^{(3)}_{j_3}
\Delta w_{j_4} \,dx dt \,\text{ and } \,   \tilde K_{j_1,j_2,j_3,j_4}=\int_0^T 
\int_\Omega  u^{(1)}_{j_1} \nabla \bar u^{(2)}_{j_2} \cdot \nabla  u^{(3)}_{j_3}
\Delta w_{j_4} \,dx dt
\end{equation}
where, again, we kept what is, at this stage, the worst distribution
of derivatives. The $\tilde J_{j_1,j_2,j_3,j_4}$ integral is nothing
but our original $I_{j_1,j_2,j_3,j_4}$ is disguise, with $\Delta
u_{j_3}$ playing the part of $u_{j_3}$ and $w$ the part of $v$; we
leave the details to the reader. The $\tilde K_{j_1,j_2,j_3,j_4}$
integral requires more care: first, remark that we may integrate by
parts twice, as the successive boundary terms will vanish due to
$u_{j_1|\partial \Omega}=0$ and then $w_{\partial\Omega}=0$. If the
Laplacian hits any individual $u_{j}$ factor, we are back with an
integral like our previous $K_{j_1,j_2,j_3,j_4}$. Hence the last
remaining case to deal with corresponds to
$$
 L_{j_1,j_2,j_3,j_4}=\int_0^T 
\int_\Omega  u^{(1)}_{j_1} \nabla^2 \bar u^{(2)}_{j_2} \nabla^2  u^{(3)}_{j_3}
 w_{j_4} \,dx dt
$$
for which we use \eqref{eq:UV} on the $u_{j_1} w_{j_4}$ factor and the
linear estimate \eqref{eq:L4lincompV} for the remaining two factors.  \cqfd
\begin{rem}
One should emphasize that
  once a global solution is obtained for $1<s<\frac 3 2$ (so that
  $H^s_0=H^s(\Omega)\cap H^1_0$), using the available Strichartz
  estimates one could propagate classical $H^\sigma(\Omega)$
  regularity for $\sigma>s$ by using the equation and energy
  estimates. The main point of the previous analysis in $X^{s}$ spaces is to illustrate that one may perform more than one
  integration by parts in the analysis of the worst case scenario
  $j_4>j_3$ (which does not exist in $\R^3$ by orthogonality due to
  disjoint Fourier supports). In fact, one could perform one last
  integration by parts (with just one gradient) in the worst part of $\tilde
  J_{j_1,j_2,j_3,j_4}$ to reach $s<4$, taking advantage of
  $\partial^2_n u_{|\partial\Omega}=0$ to dispense with the
  corresponding boundary term.
\end{rem}
Let us denote by $\tn u T =\| \mathbf{1}_{(-T,T)} u \|_{X^{1}_l}$. Proposition
\ref{p31} is a simple consequence of Banach fixed-point theorem in a
suitable ball of the space endowed with norm $\tn \cdot T$. From now on, all numerical
constants are written explicitly and numbered. We have
$$
\tn u T\leq \tn {S(t)u(0)} T + C_1 T \tn u T^3\leq C_0
\|u(0)\|_{B^{1,1}_{2,l}} + C_1 T \tn u T^3\,.
$$ 
And we may solve for 
\begin{equation}
  \label{eq:time}
  \frac{ C_2} {2\|u(0)\|_{B^{1,1}_{2_l}}^{ 2}} \leq T<
\frac{C_2}{\|u(0)\|_{B^{1,1}_{2,l}}^{ 2}}\,.
\end{equation}
We now proceed with Proposition \ref{p32}: we only need to prove an estimate on the time of existence, persistence in $X^{s}$
being a direct consequence of \eqref{eq:nonlinearXpersist}, which we
may rewrite, with $\| u\|_{s,T}=\|\mathbf{1}_{(-T,T)} u\|_{X^{s}}$,
$$
\| u\|_{s,T}\leq C_3 \|u(0)\|_{H^s}+C_4 T \tn u T^2 \|u\|_{s,T}.
$$
As such, adjusting the $C_2$ constant if necessary, 
$$
\|u(T)\|_{H^s} \leq  \| u\|_{s, T} \leq 2 C_3 \|u(0)\|_{H^s},
$$
provided $T$ satisfies \eqref{eq:time}. We now proceed with a
logarithmic Sobolev inequality: for any $f\in B^{1,1}_{2,l}$, for $s=1+\eta $, we have, for $J$ large,
\begin{align*}
  \|f\|_{B^{1,1}_{2,l}} & \leq \|S_0 f\|_2+\sum_{j\geq 0}^J (\log j)^\frac 1 2 2^j \|\Delta_j
  f\|_2+\sum_{j\geq J}(\log j)^\frac 1 2 2^j \|\Delta_j f\|_2\\
   & \leq \|S_0 f\|_2+  (J \log J)^\frac 1 2 \bigl( \sum_{j\geq 0}^J 2^{2j}
     \|\Delta_j f\|^2_2\bigr)^\frac 1 2+\sum_{j\geq J} (\log j)^\frac 1 2 2^{-\eta j}
   2^{s j} \|\Delta_j f\|_2\\
& \leq \|f\|_2+ (J\log J)^\frac 1 2 \| f\|_{H^1}+ 2^{-\frac 3 4 \eta  J} \|f\|_{H^{s}}\\
& \leq \|f\|_2+  (J\log J)^\frac 1 2 \bigl( \| f\|_{H^1}+ 2^{- \frac \eta  2 J} \|f\|_{H^{s}}\bigr)
\end{align*}
and the $(\cdots)$ is optimized in $J$ for $J=2/\eta \log(\eta
\|f\|_{H^{s}}/(2\|f\|_{H^1}))$. As such, we get
\begin{equation}
  \label{eq:logSobolev}
  \|f\|_{B^{1,1}_2}\leq \|f\|_{H^1}(C_5+C_6(\log \|f\|_{H^{s}}\log
  \log \|f\|_{H^{s}})^\frac 1 2)\,.
\end{equation}
Using this inequality, for $H^s$ data, \eqref{eq:time} yields a new
time of existence,
\begin{equation}
  \label{eq:timelog}
T\sim \frac {C_7} {\|u(0)\|^2_{H^1} \log \|u(0)\|_{H^{s}}\log \log \|u(0)\|_{H^{s}}}\,,
\end{equation}
and the proof of Proposition \ref{p32} is complete.\cqfd

Finally, we prove Theorem \ref{thcubic} by iterating the local in time
construction from Proposition \ref{p32}. In the defocusing case of \eqref{nls}, the Hamiltonian $E(u)$ is conserved
and controls the $H^1$ norm. Set $T_1=T$, and repeat the local in time
construction from the data $u(T_1)$, and reach $T_1+T_2$, with
$$
T_2\sim \frac {C_7} {E \log \|u(T_1)\|_{H^s}\log \log \|u(T_1)\|_{H^s}}\sim \frac
{C_7} {E \log (2C_3 \|u(0)\|_{H^s})\log \log (2C_3 \|u(0)\|_{H^s})}\,.
$$
after $n$ steps, we reach $T_1+T_2+\cdots+T_n$, with
$$
T_n\sim \frac {C_7} {E \log \|u(T_{n-1})\|_{H^s}\log\log  \|u(T_{n-1})\|_{H^s}}\sim \frac
{C_7} {E \log ((2C_3)^{n-1} \|u(0)\|_{H^s})\log \log ((2C_3)^{n-1} \|u(0)\|_{H^s})}
$$
which yields for $n$ large enough (depending on $C_3$ and $\|u(0)\|_{H^s}$),
$$
T_n \geq \frac {C_8} {n \log n}\,.
$$
Therefore the series $\sum_n T_n$ diverges and we may reach an
arbitrarily large time. Moreover,
$$
\|u(T_n)\|_{H^s}\leq (2C_3)^{n-1} \|u(0)\|_{H^s}\sim C_9 ( \exp \exp \exp
T_n) \|u(0)\|_{H^s}.
$$
Hence, we proved global existence in the defocusing case for
\eqref{nls}. The focusing case is handled similarly, with the
additional restriction on the mass $\|u(0)\|^2_2$, which is a
conserved quantity, and is required to be small so that
$\|u\|_{H^1}$ stays comparable to $E^\frac 1 2$ by using the
Gagliardo-Nirenberg inequality.

\end{document}